\documentclass{article}
\usepackage{graphicx,amsmath,amssymb,amsthm} 
\usepackage[margin=1in]{geometry}
\usepackage{subcaption}
\usepackage{comment}

\usepackage{algorithm}
\usepackage{algpseudocode}
\usepackage{booktabs}
\usepackage[T1]{fontenc}
\usepackage{hyperref}

\newtheorem{theorem}{Theorem}
\newtheorem{lemma}{Lemma}

\newtheorem{proposition}{Proposition}

\newcommand{\ve}{\mathcal{V}}
\title{LP-Based Algorithms for Scheduling in a Quantum Switch}
\author{R. Srikant\\ ECE, CSL, NCSA\\rsrikant@illinois.edu}

\begin{document}

\maketitle
\begin{abstract}
    We consider scheduling in a quantum switch with stochastic entanglement generation, finite quantum memories, and decoherence. The objective is to design a scheduling algorithm with polynomial-time computational complexity that stabilizes a nontrivial fraction of the capacity region. Scheduling in such a switch corresponds to finding a matching in a graph subject to additional constraints. We propose an LP-based policy, which finds a point in the matching polytope, which is further implemented using a randomized decomposition into matchings. The main challenge is that service over an edge is feasible only when entanglement is simultaneously available at both endpoint memories, so the effective service rates depend on the steady-state availability induced by the scheduling rule. To address this, we introduce a single-node reference Markov chain and derive lower bounds on achievable service rates in terms of the steady-state nonemptiness probabilities. We then use a Lyapunov drift argument to show that, whenever the request arrival rates lie within the resulting throughput region, the proposed algorithm stabilizes the request queues. We further analyze how the achievable throughput depends on entanglement generation rates, decoherence probabilities, and buffer sizes, and show that the throughput lower bound converges exponentially fast to its infinite-buffer limit as the memory size increases. Numerical results illustrate that the guaranteed throughput fraction is substantial for parameter regimes relevant to near-term quantum networking systems.
\end{abstract}

\section{Introduction}

The development of a large-scale quantum internet is widely viewed as a
key enabling technology for both secure communication and distributed
quantum computation.
Entanglement-based protocols allow distant users to establish
information-theoretically secure keys,
while also enabling remote quantum operations,
teleportation, and the interconnection of geographically distributed
quantum processors
\cite{wehner2018quantum,Muralidharan2016}.
A central architectural component in such networks is the
\emph{quantum switch},
a device that generates and stores entangled qubits and performs
entanglement swapping operations to establish end-to-end entanglement
between user pairs.
In contrast to classical packet switches,
entanglement is a fragile resource that can be stored only for limited
time due to decoherence.
As a result, scheduling decisions in a quantum switch must jointly
account for stochastic entanglement generation, finite quantum-memory
buffers, and random loss due to decoherence.

Motivated by these challenges,
recent work has begun to study scheduling policies for quantum switches. Special cases of general quantum switches were first studied in \cite{vasantam2022throughput,zubeldia2026matching,nain2022analysis,promponas2024maximizing}.
The general version of the quantum switch was studied in
\cite{Bhambay2025} where the capacity region of the
system was characterized and the throughput-optimal scheduling rules were established.
These results provide an important theoretical foundation for the
operation of quantum switches. However, 
the resulting optimal policies can be computationally expensive.
In particular, the throughput-optimal policy requires the solution of a Markov decision process, whose size grows exponentially with the number of entanglement-buffers, rendering exact dynamic programming methods impractical except for very small switches. A step towards resolving this issue has been taken in \cite{Bhambay2026}, where the authors show that performing one value iteration step at each time instant is sufficient to achieve throughput optimality. But even this could be computationally expensive if the number and size of the entanglement buffers is large. 

The quantum switch problem has some similarities with scheduling problems in traditional communication networks initiated in
\cite{TassiulasEphremides1992}, where it was shown
that MaxWeight scheduling stabilizes any arrival rate vector
in the interior of the capacity region for a broad class of constrained
queueing systems.
In the context of input-queued packet switches,
\cite{McKeown1999} demonstrated that MaxWeight-based
algorithms achieve full throughput for crossbar switches.
These results laid the foundations for a large body of work on
throughput-optimal scheduling in wireless networks and packet switches
(see, e.g., \cite{srikant2014communication}).
At the same time, a complementary line of work has investigated
\emph{constant-factor approximations} to throughput-optimal scheduling.
In many systems, implementing the exact MaxWeight rule requires solving
computationally difficult combinatorial optimization problems in every
time slot.
This observation has motivated the study of simplified scheduling rules
that achieve a guaranteed fraction of the capacity region.
For example, two well-known papers have established
half-capacity or constant-factor guarantees for switch scheduling
algorithms
\cite{weller1997scheduling,DaiPrabhakar2000}.
These results demonstrate that significant computational savings can be
achieved while sacrificing only a bounded fraction of the maximum
achievable throughput.

In this paper, we investigate similar approximation guarantees for the quantum switch.
We restrict our attention to the case where only bipartite entanglements are requested.
Our goal is to design a scheduling algorithm that
(i) has polynomial computational complexity and
(ii) provides guarantees on the fraction of the capacity region
that can be stabilized.

To this end, our contributions are as follows:
\begin{itemize}
    \item We propose an LP-based scheduling framework for the quantum
switch.
The key idea is to replace the intractable MDP with a linear programming
relaxation that determines fractional edge activation rates.
These fractional rates are then implemented through a randomized
decomposition into matchings.
A central technical challenge arises from the stochastic availability of
entanglement resources.
To address this issue, we introduce a reference entanglement-buffer
Markov chain whose steady-state availability captures the probability
that entanglement is present at each node.
Using a coupling argument, we show that the achievable service rates
under our policy are lower bounded by the product of these node
availabilities.
This leads to a \emph{coherence factor} that quantifies the
fraction of the classical matching capacity region achievable by the
quantum switch. We show that the algorithm stabilizes the queues if the arrival rates are within this fraction of the capacity region using a Lyapunov drift argument.
\item We further analyze how this coherence factor depends on physical system
parameters, including entanglement arrival rates, decoherence
probabilities, and buffer sizes.
Our analysis establishes exponential convergence of the achievable
throughput to its infinite-buffer limit and shows that the achievable
fraction of the capacity region approaches one as decoherence decreases
or entanglement generation rates increase.
\item Finally, we present a lower-complexity algorithm that omits blossom
constraints from the LP formulation.
Using classical polyhedral properties of fractional matchings,
we show that a simple scaling yields a feasible point in the matching
polytope, resulting in a scheduling policy that stabilizes a constant
fraction of the capacity region while further reducing computational
complexity.
\end{itemize}

Together, these results provide polynomial-time scheduling
algorithms with provable performance guarantees for quantum switches,
bridging ideas from classical switch scheduling,
polyhedral combinatorics, and emerging quantum networking systems.

\section{System Model and Problem Formulation}

Consider a quantum switch represented by a graph $G = (V, E)$, where $V$ is the set of vertices (quantum memories) and $E$ represents edges (user pairs). The system operates in discrete time slots $t = 0, 1, \dots$. Entangled qubits (which we will call entanglements) arrive and are queued at the vertices and service requests arrive and are queued at the edges of the graph. In each time slot, a request at an edge can be served only if an entanglement is available at each of the two vertices of the edge. 

Each node $v \in V$ maintains a local buffer of size $B_v$ to store entanglements. We denote the number of entanglements stored at vertex $v$ at time $t$ by $L_v(t) \in \{0, \dots, B_v\}$. In each time slot, new entanglements arrive at node $v$ according to a Bernoulli process with parameter $\lambda_v$.
Stored entanglements can be lost due to decoherence. We model decoherence loss as follows: each stored entanglement at node $v$ decays independently with probability $\mu_v$ per time slot.

For each edge $e \in E$, user requests arrive according to an i.i.d. stochastic process $A_e(t)$ with mean rate $\mathbb{E}[A_e(t)] = \nu_e$ and variance $\text{Var}(A_e(t)) = \sigma_e^2$. These requests are stored in an infinite-buffer queue denoted by $R_e(t)$. 

A scheduling algorithm determines which requests to serve in each time slot. A schedule is represented by a matching $M(t) \subset E$, i.e., a collection of edges such that no two edges in $M(t)$ share a common node, respecting the physical constraint that a node can serve at most one request per slot.

We assume the following order of events within a time slot:
\begin{enumerate}
    \item \textit{Schedule:} A scheduler chooses a schedule for the time slot.
    \item \textit{Service:} If an edge is included in the schedule, then if a request is waiting in the request queue at the edge, it is served if the entanglement queues at its vertices are non-empty. If a service occurs at an edge, then a request is removed from its request queue and an entanglement is removed from each of the corresponding vertex entanglement queues.
    \item \textit{Decoherence:} Next, the remaining entanglements at each vertex $u$ independently decohere with probability $\mu_u$.
    \item \textit{Arrivals:} At each vertex $u,$ a new entanglement arrives with probability $\lambda_u$. And arrivals occur at each of the edges according to their respective arrival processes.
\end{enumerate}

The queue evolution equations are given by:
\begin{align}
    R_e(t+1) &=  R_e(t) - S_e(t) + A_e(t), \\
    L_v(t+1) &= \min\left(B_v,  L_v(t) - \sum_{e \in \delta(v)} S_e(t) - D_v(t) + Y_v(t)\right),
\end{align}
where:
\begin{itemize}
    \item $S_e(t) \in \{0, 1\}$ indicates if edge $e$ was served. Service occurs only if $e\in M(t),$ $R_e(t) > 0$ and the required entanglements are available at the beginning of the slot ($L_u(t) > 0, L_v(t) > 0$). Due to the fact that each schedule is a matching in the graph and $S_e(t)$ is the actual service, $\sum_{e \in \delta(v)} S_e(t)\leq L_v(t)$ $\forall v.$ 
    \item $D_v(t)$ represents the number of entanglements lost to decoherence from $L_v(t)- \sum_{e \in \delta(v)} S_e(t)$, i.e., $D_v(t) \sim \text{Binom}(L_v(t) - \sum_{e\in\delta(v)} S_e(t), \mu_v)$
    \item $Y_v(t)$  represent the arrivals of entanglements, which occur after service and decoherence. Newly arriving entanglements are subject to decoherence only in subsequent slots.
    \item $A_e(t)$ represent the arrivals of requests, which occur after service.
\end{itemize}

\subsection{Markovian Scheduling and Throughput Optimality}

A \textit{Markovian schedule} is a randomized algorithm that selects a matching $M(t)$ based on the current state of the system, specifically the vector of queue lengths $\mathbf{Q}(t) = (\mathbf{L}(t), \mathbf{R}(t))$. As shown in \cite{Bhambay2025}, a throughput-optimal scheduling policy for such a switch can be obtained by solving a Markov Decision Process (MDP). Roughly speaking, the MDP is obtained by assuming a time-scale decomposition of the system, where the request queue states are fixed and one controls the entanglement queues by choosing schedules to maximize a long-term expected reward. Formally, this corresponds to the fluid-limit used to characterize throughput-optimal policies, where request queue lengths enter the MDP only as fixed weights \cite{Bhambay2025}. The instantaneous reward is the sum of the request queue lengths of all the edges included in the schedule that are successfully served. More precisely, the MDP is defined by the tuple $(\mathcal{S}, \mathcal{A}, \mathcal{P}, \mathcal{R})$:

\begin{itemize}
    \item \textit{State Space ($\mathcal{S}$):} The set of all possible entanglement queue configurations: $\mathcal{S} = \prod_{v \in V} \{0, \dots, B_v\}$. We will denote a specific state in the MDP by $\mathbf{s}.$
    
    \item \textit{Action Space ($\mathcal{A}$):} In each slot, the controller selects a matching $M \in \mathcal{M}$, where $\mathcal{M}$ is the set of all valid matchings in the graph.
    
    \item \textit{Reward Function ($\mathcal{R}$):} The immediate reward for taking action $M$ in state $\mathbf{s}$ is defined as the sum of the request queue lengths of the edges included in the matching:
    \begin{equation}
        r(\mathbf{s}, M) = \sum_{e \in M} w_e \cdot \mathbb{I}(\text{Resources available for } e),
    \end{equation}
    where $w_e=R_e(t)$ is fixed for the purposes of defining the MDP under a time-scale separation as in \cite{Bhambay2025}. Note that resources available for $e$ means  $(L_u(t)>0,L_v(t)>0,R_e(t)>0).$
    \item \textit{Transition Dynamics ($\mathcal{P}$):} The system evolves according to the entanglement queue evolution equations defined previously, governed by stochastic arrivals and decoherence updates.
\end{itemize}

The goal of the MDP is to find a policy $\pi$ that maximizes the long-term expected average reward. The optimal value function $W^*$ satisfies the Bellman Optimality Equation:
\begin{equation}
    W^*(\mathbf{s}) +r^*_{MDP}= \max_{M \in \mathcal{M}} \left\{ r(\mathbf{s}, M) + \sum_{\mathbf{s}' \in \mathcal{S}} P(\mathbf{s}' | \mathbf{s}, M) W^*(\mathbf{s}') \right\}
\end{equation}
where $r^*_{MDP}$ is the optimal average reward.

Solving the MDP is computationally intractable because state space size is exponential in the number of entanglement queues. For a network with $|V|$ nodes, $|\mathcal{S}|$ is $O(B^{|V|})$.  Our objective is to design a low-complexity scheduling algorithm with a provable performance guarantee. Specifically, we seek an algorithm that:
\begin{enumerate}
    \item \textit{Polynomial Complexity:} The computation time per slot should be polynomial in $|V|$ and $|E|$.
    \item \textit{Approximation Guarantee:} We aim to prove that our algorithm is $\alpha$-throughput optimal.
\end{enumerate}

\textbf{Definition ($\alpha$-Throughput Optimality):} An algorithm is $\alpha$-throughput optimal (for $0 < \alpha \le 1$) if it stabilizes the system for any arrival rate vector $\boldsymbol{\nu}$ such that $\boldsymbol{\nu} \in \alpha \Lambda$, where $\Lambda$ is the capacity region achievable by the MDP solution. In other words, the algorithm guarantees stable operation for at least an $\alpha$-fraction of the theoretical maximum throughput.

\section{Proposed Algorithm and Its Achievable Throughput}\label{sec: proposed}

As discussed in the previous section, a throughput–optimal scheduling rule can in principle be obtained by solving the Markov decision process (MDP) defined by the entanglement-buffer state $s\in\mathcal S$ and the set of feasible matchings $\mathcal M$. One way to compute the optimal average reward of a finite-state MDP is through a linear programming (LP) formulation in terms of stationary state–action frequencies.

Let $x(s,M)$ denote the steady-state fraction of time that the system is in state $s\in\mathcal S$ and action $M\in\mathcal M$ is chosen. The average-reward optimal control problem can then be written as the following LP:
\begin{align}
\text{maximize}\quad & \sum_{s\in\mathcal S}\sum_{M\in\mathcal M} r(s,M)x(s,M) \\
\text{subject to}\quad 
& \sum_{M\in\mathcal M} x(s,M)
=
\sum_{s'\in\mathcal S}\sum_{M\in\mathcal M} P(s\,|\,s',M)x(s',M), 
\quad \forall s\in\mathcal S \\
& \sum_{s\in\mathcal S}\sum_{M\in\mathcal M} x(s,M)=1 \\
& x(s,M)\ge 0, \quad \forall s,M .
\end{align}

The variables $x(s,M)$ represent the stationary occupancy measure of the MDP. The first set of constraints enforces flow conservation of the stationary distribution, while the normalization constraint ensures that the total probability mass is one. The optimal value of this LP is equal to the optimal average reward $r^*_{\text{MDP}}$.

Although this formulation is conceptually useful, it is computationally intractable for quantum switches of practical size. The state space of the MDP is 
\[
|\mathcal S|=\prod_{v\in V}(B_v+1),
\]
which grows exponentially with the number of entanglement buffers. Consequently, the above LP contains exponentially many variables and constraints.

To obtain a tractable algorithm, we therefore introduce a relaxation of this LP. Instead of tracking the full state of the entanglement queues, we consider the long-term average rate at which each edge is activated. Let
$
x_e
$
denote the long-run service attempt rate assigned to edge $e\in E$. These rates must satisfy two fundamental constraints.

First, because a node can participate in at most one entanglement swap in a given slot, the total service rate of edges incident on node $v$ cannot exceed the long-term rate at which entanglements are generated at that node. Since entanglements arrive at node $v$ with probability $\lambda_v$ per slot, this implies the node-capacity constraint
\begin{equation}
\sum_{e\in\delta(v)} x_e \le \lambda_v, \qquad \forall v\in V.
\end{equation}

Second, since only matchings can be scheduled in any given slot, the vector $x=(x_e)_{e\in E}$ must lie in the convex hull of the set of matchings of the graph. This leads to the following linear programming relaxation of the MDP control problem:
\begin{align}
    \text{maximize} \quad & \sum_{e \in E} w_e(t) x_e \label{eq: LP1}\\
    \text{subject to} \quad & \sum_{e \in \delta(v)} x_e \le \lambda_v, \quad \forall v \in V \label{eq: LP2}\\
    \quad & \sum_{e \in E(S)} x_e \;\le\; \frac{|S|-1}{2},
\quad \forall\, S \subseteq V \text{ with } |S| \text{ odd},\ |S|\ge 3 \label{eq: blossom}\\
    & x_e\geq 0, \quad \forall e \in E, \label{eq: LP3}
\end{align}
where $E(S)$ denotes the set of all edges in $S.$

We note that (\ref{eq: blossom}) are blossom inequalities used to characterize the matching polytope \cite{Edmonds65}. However, due to (\ref{eq: LP2}), our constraint set is not a matching polytope, but a polytope which is a subset of the matching polytope since $\lambda_v\leq 1$ $\forall v$.
The optimal solution $\mathbf{x}^*$ to the LP defined by \eqref{eq: LP1}--\eqref{eq: LP3} represents an upper bound on the maximum value achievable under the assumption that the switch can support fractional edge activations, i.e., the optimal solution $\mathbf{x}^*$ may be such that, for each edge $e,$ $x_e^*$ may not lie in $\{0,1\}$ but will more generally lie in $[0,1]$. 

Since the vector $\mathbf{x}^*$ is guaranteed to lie within the matching polytope of the graph $G$, defined as the convex hull of all valid integral matchings, by Carathéodory's theorem, $\mathbf{x}^*$ admits a decomposition into a convex combination of at most $|E|+1$ integral matchings. This decomposition yields a set of matchings $\mathcal{M} = \{M_1, M_2, \dots, M_k\}$ and associated probabilities $\{p_1, p_2, \dots, p_k\}$ such that:
\begin{equation}\label{eq: decomp}
    \sum_{j=1}^k p_j \mathbf{1}_{M_j} = \mathbf{x}^* \quad \text{and} \quad \sum_{j=1}^k p_j = 1, \qquad p_j\geq 0 \quad\forall j.
\end{equation}

\paragraph{Proposed Scheduling Algorithm} Our scheduling algorithm is as follows: at time slots $t=0, T, 2T,...$ observe the request queue length vector $R(t).$ Obtain probabilities $\{p_M\}$ as above and select matchings according to these probabilities over time slots $t,t+1,t+T-1.$ In other words, a probabilistic schedule is chosen once every $T$ slots using the request queue lengths at the beginning of the $T$ slots. 

We will show that the parameter $T$ can be chosen appropriately depending on how far the arrival rates are from the boundary of the capacity region.

\subsection{Solving the LP and Decomposing the Solution Into Matchings}

The LP in \eqref{eq: LP1}--\eqref{eq: LP3} can be solved in polynomial time using
the ellipsoid method \cite{grotschel1988geometric} due to the fact that there is
a separation oracle for the constraints \cite{padberg1982odd,letchford2004faster}.
Given the solution $\mathbf{x}^*$, a polynomial-time algorithm to compute an
explicit convex decomposition into matchings is provided in~\cite{CarrVempala02}.
However, as in much of LP theory, algorithms that are not theoretically
polynomial-time are often practically more efficient. For our problem, this
means that the LP should be solved using the simplex method and the decomposition
should be performed using column generation \cite{DesrosiersLubbecke2005}, which
we present in Algorithm~\ref{alg:decomp_colgen}.

\begin{algorithm}
\caption{Decomposition Using Column Generation}
\label{alg:decomp_colgen}
\begin{algorithmic}[1]
\Require Vector $\mathbf{x}^* \in Co(\mathcal{M})$
\Ensure Set of pairs $\{(p_k, M_k)\}$ such that
$\sum_k p_k \mathbf{1}_{M_k} = \mathbf{x}^*$,
$\sum_k p_k = 1$, and $p_k \ge 0$

\State \textbf{Initialize:} Choose any nonempty set
\(\mathcal K\subseteq\mathcal M\), e.g.,
a set of matchings obtained heuristically.
\Loop
    \State \textit{// Step 1: Phase-I Restricted Master Problem}
    \State Solve the following LP over the current set $\mathcal{K}$:
    \[
    \begin{aligned}
        \text{minimize} \quad
            & \sum_{e\in E} (a_e^+ + a_e^-) \\
        \text{subject to} \quad
            & \sum_{M\in\mathcal{K}} p_M \mathbf{1}_M
              + \mathbf{a}^+ - \mathbf{a}^- = \mathbf{x}^*, \\
            & \sum_{M\in\mathcal{K}} p_M = 1, \\
            & p_M \ge 0 \quad \forall M\in\mathcal{K}, \\
            & a_e^+ \ge 0,\quad a_e^- \ge 0 \quad \forall e\in E .
    \end{aligned}
    \]
    \State Let $(\mathbf{y},z)$ be the dual variables associated with the
    equality constraints.

    \State \textit{// Step 2: Pricing Oracle}
    \State $M^* \leftarrow \text{MaxWeightMatching}(G,\text{weights}=\mathbf{y})$
    \State $W^* \leftarrow \sum_{e\in M^*} y_e$

    \State \textit{// Step 3: Check Termination Condition}
    \If{$W^* \le -z$}
    \State \textbf{Break}
    \Comment{No improving column exists; since $\mathbf{x}^*\in Co(\mathcal{M})$, the Phase-I objective is zero}
\Else
    \State $\mathcal{K} \leftarrow \mathcal{K}\cup\{M^*\}$
\EndIf
\EndLoop
\State \Return $\{(p_M,M)\mid M\in\mathcal{K},\ p_M>0\}$
\end{algorithmic}
\end{algorithm}

By introducing artificial variables
\(\mathbf a^+\) and \(\mathbf a^-\), the problem of decomposing edge activation probabilities into matchings can be written as
\begin{align*}
    \text{minimize} \quad
        & \sum_{e\in E}(a_e^+ + a_e^-) \\
    \text{subject to} \quad
        & \sum_{M\in\mathcal M}p_M\mathbf 1_M
          +\mathbf a^+-\mathbf a^-=\mathbf x^*, \\
        & \sum_{M\in\mathcal M}p_M=1, \\
        & p_M\ge 0 \quad \forall M\in\mathcal M, \\
        & a_e^+\ge 0,\quad a_e^-\ge 0 \quad \forall e\in E .
\end{align*}
Moreover, since
\(\mathbf x^*\in Co(\mathcal M)\), the above LP, which we will call the full Master LP,  has an optimal value
zero.

The dual of this LP is
\begin{align*}
    \text{maximize} \quad
        & \mathbf y\cdot \mathbf x^* + z \\
    \text{subject to} \quad
        & \mathbf y\cdot \mathbf 1_M + z \le 0
          \quad \forall M\in\mathcal M, \\
        & y_e \le 1 \quad \forall e\in E, \\
        & -y_e \le 1 \quad \forall e\in E .
\end{align*}
The bounds \(y_e\le 1\) and \(-y_e\le 1\) arise from the artificial variables
\(a_e^+\) and \(a_e^-\), respectively.

Since the full master LP has one variable \(p_M\) for each matching
\(M\in\mathcal M\), it has exponentially many variables. Column generation
therefore begins with a restricted set of matchings \(\mathcal K\subseteq
\mathcal M\), and solves the restricted Phase-I master obtained by replacing
\(\mathcal M\) by \(\mathcal K\):
\begin{align*}
    \text{minimize} \quad
        & \sum_{e\in E}(a_e^+ + a_e^-) \\
    \text{subject to} \quad
        & \sum_{M\in\mathcal K}p_M\mathbf 1_M
          +\mathbf a^+-\mathbf a^-=\mathbf x^*, \\
        & \sum_{M\in\mathcal K}p_M=1, \\
        & p_M\ge 0 \quad \forall M\in\mathcal K, \\
        & a_e^+\ge 0,\quad a_e^-\ge 0 \quad \forall e\in E .
\end{align*}
Let \(\mathbf y\in\mathbb R^{|E|}\) and \(z\in\mathbb R\) be the dual
variables associated with the equality constraints. The dual of the restricted
Phase-I master is
\begin{align*}
    \text{maximize} \quad
        & \mathbf y\cdot \mathbf x^* + z \\
    \text{subject to} \quad
        & \mathbf y\cdot \mathbf 1_M + z \le 0
          \quad \forall M\in\mathcal K, \\
        & y_e \le 1 \quad \forall e\in E, \\
        & -y_e \le 1 \quad \forall e\in E .
\end{align*}

Thus the restricted dual differs from the full dual only in that it contains
the matching constraints
\[
    \mathbf y\cdot \mathbf 1_M+z\le 0
\]
only for \(M\in\mathcal K\), rather than for all \(M\in\mathcal M\). The
pricing oracle checks whether any omitted matching constraint is violated.
This amounts to solving
\[
    W^*=\max_{M\in\mathcal M}\mathbf y\cdot \mathbf 1_M,
\]
which is precisely a maximum weight matching problem in \(G\) with edge weights
\(\mathbf y\). If \(W^*>-z\), then the maximizing matching \(M^*\) violates the
dual constraint
\[
    \mathbf y\cdot \mathbf 1_{M^*}+z\le 0,
\]
and we add \(M^*\) to \(\mathcal K\).

If \(W^*\le -z\), then
\[
    \mathbf y\cdot \mathbf 1_M+z\le 0
    \qquad \forall M\in\mathcal M.
\]
Hence the solution of the restricted dual is feasible for the full dual. Since
it is optimal for the restricted dual, and since the restricted dual is a
relaxation of the full dual, the restricted and full Phase-I optimal values are
equal. But the full Phase-I optimal value is zero because
\(\mathbf x^*\in Co(\mathcal M)\). Therefore the restricted Phase-I master also
has optimal value zero at termination. Consequently, the artificial variables
vanish and the remaining probabilities \(p_M\) satisfy
\[
    \sum_{M\in\mathcal K}p_M\mathbf 1_M=\mathbf x^*,
    \qquad
    \sum_{M\in\mathcal K}p_M=1,
    \qquad
    p_M\ge 0.
\]
Thus Algorithm~\ref{alg:decomp_colgen} returns a valid convex decomposition of
\(\mathbf x^*\) into matchings.

\subsection{Achievable Throughput}\label{sec: achievable}

For the algorithm to be able to provide service to a request queue, we need to have an entanglement in each of the entanglement queues at the vertices of the edge corresponding to the request queue. To characterize the probability with which this event happens, 
we consider a vertext $u$ and define a reference discrete-time Markov chain $\{\tilde{L}_u(t)\}_{t \ge 0}$ with state space $\{0, 1, \dots, B_u\}$ and a fixed service attempt probability:
$p_u = \lambda_u.$
The transition probabilities $P_{ij}$ for the Markov chain are derived by considering its state evolution through the three phases: service, decoherence and then arrivals, in that order. Let $L$ be the state at the start of the slot. 
First, the intermediate state after service is $K_{serv} = \max(0, L - S)$, where $S \in \{0,1\}$ is the service indicator. 
Second, the number of entanglements surviving decoherence is $K_{dec} \sim \text{Binom}(K_{serv}, 1-\mu_u)$. 
Finally, the state at the start of the next slot is $j = \min(B_u, K_{dec} + A)$, where $A \in \{0,1\}$ is the arrival indicator. 

The probability of transitioning from state $i$ to $j$ is given by:
\begin{equation}
    P_{ij} = \sum_{S=0}^1 \sum_{A=0}^1 \mathbb{P}(S)\mathbb{P}(A) \sum_{K_{dec}} \binom{\max(0, i-S)}{K_{dec}} (1-\mu_u)^{K_{dec}} \mu_u^{\max(0, i-S)-K_{dec}} \cdot \mathbb{I}(j = \min(B_u, K_{dec} + A))
\end{equation}
where $\mathbb{P}(S=1) = p_u$ and $\mathbb{P}(A=1) = \lambda_u$. if $\pi$ is the vector of stationary probabilities for this Markov chain, then the stationary probability of a entanglement being available is $C_u = 1 - \pi_0$.

\begin{lemma}[Availability Lower Bound via Coupling]
\label{lemma:coupling_bound}
Let $\mathbb{P}(L_u > 0, L_v > 0)$ be the stationary probability that the buffers at both endpoints of edge $e=(u,v)$ are non-empty under the randomized decomposition policy with fixed weights $\{w_e\}$. Let $C_u$ and $C_v$ be the steady-state availabilities of the reference chains $\tilde{L}_u$ and $\tilde{L}_v$ as defined above. Then:
\begin{equation}
    \mathbb{P}(L_u > 0, L_v > 0) \ge (C_u +C_v-1)^+.
\end{equation}
\end{lemma}

\begin{proof}
Fix a node $u$. Under our randomized decomposition policy, the probability that
$u$ is scheduled in a given slot is
\[
a_u := \sum_{e\in\delta(u)} x_e^* \le \lambda_u =: p_u,
\]
where the inequality follows from the LP constraints.

For each node $u$ and slot $t$, let $U_u(t)\sim{\rm Unif}[0,1]$ be i.i.d. over $u,t$,
independent of arrivals and decoherence. Define Bernoulli indicators
\[
S_u(t):=\mathbb{I}\{U_u(t)\le a_u\},\qquad
\tilde S_u(t):=\mathbb{I}\{U_u(t)\le p_u\}.
\]
Then $S_u(t)\le \tilde S_u(t)$ almost surely for all $t$. Construct the reference
chain $\tilde L_u(t)$ to evolve with the \emph{same} arrival and decoherence events
as $L_u(t)$ at node $u$, but with removals driven by $\tilde S_u(t)$ (i.e., whenever
$\tilde S_u(t)=1$ and $\tilde L_u(t)>0$, one entanglement is removed before decoherence).
Since the actual process can remove at most one entanglement from $u$ per slot, and $S_u(t)\le\tilde S_u(t)$
pathwise, this coupling implies the sample-path dominance
\[
L_u(t)\ge \tilde L_u(t)\qquad \forall t.
\]
Thus, $P(L_u>0)\geq C_u.$ The result follows from the fact $$1\geq P(L_u>0\,\, \mathrm{or}\,\,  L_v>0)=P(L_u>0)+P(L_v>0)-P(L_u>0,L_v>0).$$
\end{proof}

Next, as in \cite{Bhambay2025}, we define a request-agnostic policy $\pi$ as one that selects matching $M$ with probability $\pi(M|\mathbf{L})$ when the entanglement queue length vector is in state $\mathbf{L}.$ Let $d_\pi$ be the stationary distribution of the entanglement queues under such a policy. Recall that $\mathbf{\nu}$ is the vector of request arrival rates. Following \cite{Bhambay2025}, the capacity region of the switch is defined as follows: 
$$\mathcal{C}=\{\mathbf{\nu}: \exists \pi \textrm{ such that } \nu_e <\sum_{\mathbf{L}:l_{u_e}>0, l_{v_e}>0} \sum_{M\ni e} \pi(M|\mathbf{L})d_\pi(L)\quad \forall e\in E\},
$$
where we have used the notation $u_e, v_e$ to denote the vertices associated with edge $e.$ Define
$
\Gamma_{\mathrm{coh}} := \min_{e=(u,v)\in E}( C_u +C_v -1)^+.
$
Next, we show that our proposed algorithm achieves at least $\Gamma_{\mathrm{coh}}$ fraction of the capacity region. Note that $\Gamma_{\mathrm{coh}}$ can potentially be zero for certain values of the problem parameters. Later, we will demonstrate using a combination of theory and simulations that for switch parameters of interest, $\Gamma_{\mathrm{coh}}$ is a non-trivial fraction.

\begin{theorem}\label{thm: throughput}
Suppose that $\frac{(1+\epsilon)}{\Gamma_{\text{coh}}}\mathbf{\nu} \in \mathcal{C}$ and consider the scheduling algorithm proposed at the end of Section~\ref{sec: proposed}. Define the Lyapunov function $$\ve(\mathbf{R})=\frac{1}{2}\sum_e R_e^2.$$ Then, $\exists T_\epsilon, \delta>0$ and a finite set $\mathcal{B}$ such that $\forall k\in\{0,1,2,\ldots\}$
$$E\left(\ve(\mathbf{R}((k+1)T_\epsilon))-\ve(\mathbf{R}(kT_\epsilon))|\mathbf{L}(kT_\epsilon),\mathbf{R}(kT_\epsilon)\right)\leq -\delta$$
whenever $R(kT_\epsilon)\in \mathcal{B}^c.$
\end{theorem}
\begin{proof}
We will present the proof in several steps.

\textit{Step 1: Uniform Mixing of Entanglement Queues.} 
We use the uniform mixing property of the entanglement queues proved in \cite{Bhambay2025}. For any fixed matching policy $\pi$, the entanglement process $\{\mathbf{L}(t)\}$ is a Markov chain on the finite state space $\mathcal{S} = \prod_v\{0, \dots, B_v\}$. Let $P_\pi$ be its transition matrix. Since $\mu_v, \lambda_v \in (0,1)$, the state $\mathbf{0}$ (all buffers empty) is reachable from any state $\mathbf{L} \in \mathcal{S}$ in one time slot via the the decoherence of all entanglements after service and the absence of arrivals. Specifically, for all $\pi \in \Pi$ and $\mathbf{L} \in \mathcal{S}$:
\begin{equation}
    P_\pi(\mathbf{L}, \mathbf{0}) \ge \prod_{v \in V} (1-\lambda_v)\mu_v^{B_v} \triangleq \eta > 0,
\end{equation}
independent of $\pi.$ Thus, by the standard Doeblin condition, the chain is uniformly ergodic with a contraction coefficient $\rho :=1-\eta < 1$. Consequently,  $\|P_\pi^t(\mathbf{L}, \cdot) - d_\pi(\cdot)\|_{TV} \le \rho^t$ $\forall \mathbf{L}, t\geq 0, \pi.$

\textit{Step 2: $T$-Step Drift Decomposition.}
Since departures occur before arrivals and $S_e(t)$ denotes \emph{actual} service (departures),
we have $S_e(t)\le R_e(t)$ a.s., and hence the request queue evolves as
\[
R_e(t+1)=R_e(t)-S_e(t)+A_e(t),
\]
without the need for a projection operator.

Fix a frame index $k$ and a frame size $T$. Iterating over the frame yields
\[
R_e((k+1)T)
=
R_e(kT)+\sum_{\tau=kT}^{(k+1)T-1}\big(A_e(\tau)-S_e(\tau)\big).
\]
Squaring both sides and subtracting $R_e^2(kT)$ gives
\[
R_e^2((k+1)T)-R_e^2(kT)
=
2R_e(kT)\sum_{\tau=kT}^{(k+1)T-1}\big(A_e(\tau)-S_e(\tau)\big)
+
\Big(\sum_{\tau=kT}^{(k+1)T-1}\big(A_e(\tau)-S_e(\tau)\big)\Big)^2 .
\]
Summing over all edges and taking conditional expectations given $(\mathbf{L}(kT),\mathbf{R}(kT))$,
the $T$-step Lyapunov drift
\[
\Delta_T :=
\mathbb E\!\left[\mathcal V(\mathbf R((k+1)T))-\mathcal V(\mathbf R(kT))\mid \mathbf L(kT),\mathbf R(kT)\right]
\]
satisfies
\begin{equation}\label{eq:Tstep-drift-basic}
\Delta_T
\le
\sum_e R_e(kT)\,
\mathbb E\!\left[\sum_{\tau=kT}^{(k+1)T-1}\big(A_e(\tau)-S_e(\tau)\big)\right]
+
\frac12\sum_e
\mathbb E\!\left[\Big(\sum_{\tau=kT}^{(k+1)T-1}\big(A_e(\tau)-S_e(\tau)\big)\Big)^2\right].
\end{equation}

By Cauchy--Schwarz,
\[
\Big(\sum_{\tau=kT}^{(k+1)T-1}(A_e(\tau)-S_e(\tau))\Big)^2
\le
T\sum_{\tau=kT}^{(k+1)T-1}(A_e(\tau)-S_e(\tau))^2.
\]
Since $S_e(\tau)\in\{0,1\}$ and $\mathbb E[A_e^2]=\sigma_e^2+\nu_e^2$, we have
\[
\mathbb E[(A_e(\tau)-S_e(\tau))^2]
\le
2\mathbb E[A_e(\tau)^2]+2\mathbb E[S_e(\tau)^2]
\le
2(\sigma_e^2+\nu_e^2)+2.
\]
Therefore,
\[
\frac12\sum_e
\mathbb E\!\left[\Big(\sum_{\tau=kT}^{(k+1)T-1}(A_e(\tau)-S_e(\tau))\Big)^2\right]
\le
K_T,
\qquad
K_T:=\frac{T^2}{2}\sum_e\big(2(\sigma_e^2+\nu_e^2)+2\big)=O(T^2).
\]

Define the \emph{potential} service indicator
\[
\hat S_e(t):=\mathbf 1\{e\in M(t),\,L_{u_e}(t)>0,\,L_{v_e}(t)>0\},
\]
which ignores request-queue emptiness. Then $S_e(t)=\hat S_e(t)\mathbf 1\{R_e(t)>0\}\le \hat S_e(t)$.
Moreover, since $S_e(t)\le 1$ and arrivals are nonnegative, we have the deterministic bound
\[
R_e(kT+\tau)\ge R_e(kT)-\tau,\qquad \tau=0,1,\dots,T-1.
\]
Hence, if $R_e(kT)\ge T$, then $R_e(kT+\tau)>0$ for all $\tau\in\{0,\dots,T-1\}$ and therefore
$S_e(kT+\tau)=\hat S_e(kT+\tau)$ throughout the frame. It follows that for every edge $e$,
\begin{equation}\label{eq:S-vs-Shat-sum}
\sum_{\tau=0}^{T-1} S_e(kT+\tau)
\ge
\sum_{\tau=0}^{T-1} \hat S_e(kT+\tau)
-
T\,\mathbf 1\{R_e(kT)<T\}.
\end{equation}
Multiplying \eqref{eq:S-vs-Shat-sum} by $R_e(kT)$ and using $R_e(kT)\mathbf 1\{R_e(kT)<T\}\le T$ gives
\begin{equation}\label{eq:weighted-S-vs-Shat}
R_e(kT)\sum_{\tau=0}^{T-1} S_e(kT+\tau)
\ge
R_e(kT)\sum_{\tau=0}^{T-1} \hat S_e(kT+\tau)
-
T^2.
\end{equation}
Summing \eqref{eq:weighted-S-vs-Shat} over $e$ yields an additive penalty of at most $|E|T^2$,
which can be absorbed into the $O(T^2)$ term $K_T$.

Conditional on $\mathbf R(kT)$, the distribution $\{p_M\}$ is fixed over the frame.
Thus, for $\tau'=0,1,\dots,T-1$,
\[
\mathbb P(\hat S_e(kT+\tau')=1 \mid \mathbf L(kT),\mathbf R(kT))
=
\sum_{M\ni e} p_M\,
\mathbb P(L_{u_e}(kT+\tau')>0,L_{v_e}(kT+\tau')>0 \mid \mathbf L(kT)).
\]
Let
\[
\mu_e(\mathbf R(kT)):=\sum_{M\ni e} p_M\, d_\pi(L_{u_e}>0,L_{v_e}>0)
\]
denote the corresponding stationary potential service probability under the entanglement-chain stationary
distribution $d_\pi$ induced by these fixed probabilities $\{p_M\}$.
By the uniform mixing bound from Step~1, there exists a constant $C'>0$ such that, uniformly over $\mathbf L(kT)$,
\[
\sum_{\tau'=0}^{T-1}\mathbb E[\hat S_e(kT+\tau')\mid \mathbf L(kT),\mathbf R(kT)]
\ge
T\,\mu_e(\mathbf R(kT)) - C'.
\]
Combining this with \eqref{eq:weighted-S-vs-Shat} and summing over $e$ yields
\[
\sum_e R_e(kT)\,\mathbb E\!\left[\sum_{\tau=kT}^{(k+1)T-1} S_e(\tau)\right]
\ge
T\sum_e R_e(kT)\mu_e(\mathbf R(kT))
-
C'\sum_e R_e(kT)
-
|E|T^2.
\]
Substituting the above bound (and $\mathbb E[\sum_{\tau}A_e(\tau)]=T\nu_e$) into \eqref{eq:Tstep-drift-basic}, and, abusing notation and absorbing the additional $|E|T^2$ term into $K_T$, we obtain the drift bound
\[
\Delta_T
\le
T\sum_e R_e(kT)\big(\nu_e-\mu_e(\mathbf R(kT))\big)
+
C'\sum_e R_e(kT)
+
K_T,
\]
where $K_T=O(T^2)$ is independent of $\mathbf L(kT)$ and $\mathbf R(kT)$.

\textit{Step 3: Lower bounding the policy service rate.}
Fix a frame $k$ and consider the matching distribution $\{p_M\}$ computed at time $kT$.
Let
\[
\mu_e(\mathbf R(kT))
:=
\sum_{M\ni e} p_M\, d_\pi(L_{u_e}>0,L_{v_e}>0)
\]
denote the stationary \emph{potential} service probability for edge $e=(u_e,v_e)$
under the entanglement-chain stationary distribution $d_\pi$ induced by $\{p_M\}$.
By construction of the decomposition in Section~\ref{sec: proposed},
\[
\sum_{M\ni e} p_M = x^*_e,
\]
where $\mathbf x^*$ is the optimal solution of the LP
\eqref{eq: LP1}--\eqref{eq: LP3} with weights $\mathbf R(kT)$.
By Lemma~\ref{lemma:coupling_bound},
$d_\pi(L_{u_e}>0,L_{v_e}>0)\ge (C_{u_e}+C_{v_e}-1)^+\ge \Gamma_{\mathrm{coh}}$,
and therefore
\begin{equation}\label{eq:mu-lower-step3}
\mu_e(\mathbf R(kT))
\ge
\Gamma_{\mathrm{coh}}\,x^{*}_{e}.
\end{equation}
Multiplying \eqref{eq:mu-lower-step3} by $R_e(kT)$ and summing over $e$ yields
\begin{equation}\label{eq:weighted-mu-lower-step3}
\sum_e R_e(kT)\,\mu_e(\mathbf R(kT))
\ge
\Gamma_{\mathrm{coh}}
\sum_e R_e(kT)\,x^{*}_{e}.
\end{equation}

\medskip
\textit{Step 4: Relating the relaxed LP value to the arrival rates.}
By the assumption of the theorem,
\[
\Gamma_{\mathrm{coh}}(1+\epsilon)\boldsymbol{\nu}\in\mathcal C.
\]
Hence there exists a request-agnostic stationary policy $\pi^\dagger$ whose
steady-state \emph{successful} service rates
$\mathbf s^\dagger=\{s^\dagger_e\}_{e\in E}$ satisfy
\[
s^\dagger_e \ge \frac{(1+\epsilon)}{\Gamma_{\mathrm{coh}}}\nu_e,
\qquad \forall e\in E.
\]
The vector $\mathbf s^\dagger$ represents the steady-state successful
service rates under policy $\pi^\dagger$. In each slot, the set of
successfully served edges forms a matching; therefore the long-run
average service rate vector lies in the convex hull of matchings,
i.e., $\mathbf s^\dagger \in Co(\mathcal M)$.
Moreover, since at most one entanglement can be consumed per node per slot,
we have $\sum_{e\in\delta(v)} s^\dagger_e \le \lambda_v$ for all $v$.
Hence $\mathbf s^\dagger$ is feasible for the LP
\eqref{eq: LP1}--\eqref{eq: LP3}.
By optimality of
$\mathbf x^*$ for weights $\mathbf R(kT)$,
\begin{equation}\label{eq:rel-dominates-sdagger-step4}
\sum_e R_e(kT)\,x^{*}_{e}
\ge
\sum_e R_e(kT)\,s^\dagger_e
\ge
\frac{(1+\epsilon)}{\Gamma_{\mathrm{coh}}}
\sum_e R_e(kT)\,\nu_e.
\end{equation}

\medskip
\textit{Step 5: Completing the drift argument.}
Combining \eqref{eq:weighted-mu-lower-step3} and
\eqref{eq:rel-dominates-sdagger-step4} gives
\[
\sum_e R_e(kT)\,\mu_e(\mathbf R(kT))
\ge
(1+\epsilon)\sum_e R_e(kT)\,\nu_e.
\]
Substituting this bound into the $T$-step drift inequality from Step~2,
\[
\Delta_T
\le
T\sum_e R_e(kT)\big(\nu_e-\mu_e(\mathbf R(kT))\big)
+
C'\sum_e R_e(kT)
+
K_T,
\]
yields
\[
\Delta_T
\le
-\epsilon T\sum_e R_e(kT)\,\nu_e
+
C'\sum_e R_e(kT)
+
K_T.
\]
Let $\nu_{\min}:=\min\{\nu_e:\nu_e>0\}$.
Then $\sum_e R_e\nu_e\ge \nu_{\min}\sum_{e:\nu_e>0}R_e$.
Choosing $T_\epsilon$ such that $\epsilon T_\epsilon\nu_{\min}>C'$,
there exist $\gamma,\delta>0$ and a finite set $\mathcal B$ for which
$\Delta_{T_\epsilon}\le -\delta$ whenever $\mathbf R(kT_\epsilon)\notin\mathcal B$.
\end{proof}

The above theorem can then be used to show positive recurrence of the Markov chain or other notions of stability used in the literature using the Foster-Lyapunov theorem and related techniques \cite{srikant2014communication}. Further, the Lyapunov drift can also be useful to obtain  bounds on queue lengths as in \cite{eryilmaz2012asymptotically,maguluri2016heavy}. We note that, instead of using a $T_\epsilon$ that depends on $\epsilon,$ one can also choose an adaptive frame size based on the queue length as proposed in \cite{Bhambay2026}. The same throughput guarantees would continue to hold following arguments similar to \cite{Bhambay2026}.

In the next subsection, we will explore how $\Gamma_{\mathrm{coh}}$ varies as a function of the switch parameters: the arrival rates of entanglements, the decoherence probabilities and buffer sizes.

\subsection{A Lower Complexity Algorithm}\label{sec: lower}
\label{sec:simpler}

 Now, we preent a lower complexity algorithm. We onsider the degree-only LP relaxation, without the blossom inequalities:
\begin{align}
\text{maximize}\quad & \sum_{e\in E} w_e(t)\,x_e \label{eq:simpleLP1}\\
\text{subject to}\quad & \sum_{e\in\delta(v)} x_e \le \lambda_v,\qquad \forall v\in V \label{eq:simpleLP2}\\
& x_e\ge 0,\qquad \forall e\in E. \label{eq:simpleLP3}
\end{align}
Let $\mathbf x^{\rm frac}$ denote an optimal solution. In general, $\mathbf x^{\rm frac}$
need not lie in the matching polytope $Co(\mathcal M)$ because the blossom inequalities
may be violated. Nevertheless, a standard argument in polyhedral combinatorics implies that the scaled vector
\[
\mathbf x^{\rm app}:=\tfrac23\,\mathbf x^{\rm frac}
\]
is feasible for the matching polytope, i.e., $\mathbf x^{\rm app}\in Co(\mathcal M)$.
While this appears to be a well known result, a proof does not appear in commonly used texts in the field; hence, we sketch it here for completeness: 
\begin{itemize}
    \item The extreme points of the degree-only
fractional matching polytope are half-integral, and the edges with value $1/2$
form node-disjoint odd cycles \cite{balinski1970maximum,balas1981integer}.
\item On any odd cycle of length $2k+1$, the unscaled solution has total weight $(2k+1)/2$
on the cycle edges, which violates the odd-set (blossom) bound $k$; scaling by $2/3$
reduces this to $(2k+1)/3\le k$, hence all blossom inequalities
\[
\sum_{e\in E(S)} x_e \le \frac{|S|-1}{2},\qquad \forall S\subseteq V\ \text{odd},\ |S|\ge 3
\]
are satisfied, and therefore $\mathbf x^{\rm app}\in Co(\mathcal M)$ \cite{Edmonds65}.
\end{itemize}

Since $\mathbf x^{\rm app}\in Co(\mathcal M)$, it admits a convex decomposition into
integral matchings, and we may implement the same randomized matching schedule as in
Section~\ref{sec: proposed}, with $\mathbf x^*$ replaced by $\mathbf x^{\rm app}$.
The resulting node activation probabilities satisfy
\[
a_v:=\sum_{e\in\delta(v)} x^{\rm app}_e \le \tfrac23\,\lambda_v,
\]
so the coupling argument of Lemma~\ref{lemma:coupling_bound} applies with reference
service-attempt probability $p_v=\tfrac23\,\lambda_v$, yielding a coherence factor
$\Gamma'_{\rm coh}$ defined analogously.

The main advantage of the simpler algorithm is computational: 
\eqref{eq:simpleLP1}--\eqref{eq:simpleLP3} contains only the $|V|$ degree
constraints,
and can therefore be solved efficiently in practice using standard LP solvers. 
Let $\Gamma'_{\mathrm{coh}}$ denote the coherence factor
computed using the reference entanglement chain with service-attempt
probability $p_v=\tfrac23\,\lambda_v$ at each node $v$.
Then, by the same Lyapunov-drift argument as in
Theorem~\ref{thm: throughput}, the simpler algorithm stabilizes
all arrival-rate vectors $\boldsymbol{\nu}$ satisfying
\[
\frac{(1+\epsilon)}{\tfrac23\,\Gamma'_{\mathrm{coh}}}\,\boldsymbol{\nu}
\in \mathcal C.
\]
In other words, this algorithm guarantees stability for a constant
$\tfrac23\Gamma_{\mathrm{coh}}'$  of the
capacity region $\mathcal C$ defined in
Section~\ref{sec: achievable}. 

If
\[
    \sum_{v\in S}\lambda_v \le |S|-1
    \qquad
    \text{for every odd set } S\subseteq V,\ |S|\ge 3,
\]
then the blossom inequalities are redundant. In particular, if
\(\lambda_v\le 2/3\) for all \(v\in V\), then all blossom inequalities are
implied by the node constraints.\footnote{We thank Chandra Chekuri for the observation that blossom inequalities are redundant for sufficiently small $\lambda_v.$}

To see this, fix an odd set \(S\subseteq V\). Summing the node constraints over
all \(v\in S\), we obtain
\[
    \sum_{v\in S}\sum_{e\in\delta(v)}x_e
    \le
    \sum_{v\in S}\lambda_v .
\]
Each edge in \(E(S)\) is counted twice in the left-hand side, while edges with
exactly one endpoint in \(S\) are counted once. Since \(x_e\ge 0\), it follows
that
\[
    2\sum_{e\in E(S)}x_e
    \le
    \sum_{v\in S}\sum_{e\in\delta(v)}x_e
    \le
    \sum_{v\in S}\lambda_v .
\]
Therefore,
\[
    \sum_{e\in E(S)}x_e
    \le
    \frac12\sum_{v\in S}\lambda_v .
\]
If \(\sum_{v\in S}\lambda_v\le |S|-1\), then
\[
    \sum_{e\in E(S)}x_e
    \le
    \frac{|S|-1}{2},
\]
which is precisely the blossom inequality for \(S\). Finally, if
\(\lambda_v\le 2/3\) for all \(v\), then for every odd set \(S\) with
\(|S|\ge 3\),
\[
    \sum_{v\in S}\lambda_v
    \le
    \frac{2}{3}|S|
    \le
    |S|-1,
\]
where the last inequality holds because \(|S|\ge 3\). Thus the blossom
inequalities are redundant under the uniform sufficient condition
\(\lambda_v\le 2/3\) for all \(v\). In our numerical results, considering technological constraints in the near future, we will only consider $\lambda_v\leq 2/3$ $\forall v.$ Thus, for practical purposes, one can use the lower complexity algorithm in this section and achieve $\Gamma_{coh}$ fraction of the capacity region.

\section{Throughput Performance as a function of Switch Parameters}\label{sec: throughput}

As technology improves, it is expected that entanglement arrival rates will increase and decoherence rates will decrease. Further, buffer sizes are also expected to increase over time. In this subsection, we will consider a single node and study its availability $C_u$ as a function of the switch parameters $\lambda_u, \mu_u, B_u$. Since $\Gamma_{\mathrm{coh}}$ is going to be determined by the vertices with smallest $C_u$ in the switch, we will consider one vertex and drop the subscript $u$ in this section. Additionally, we will add a superscript $C^B$ to indicate the lower bound on availability at a vertex with buffer size $B.$ Motivated by results for traditional switches in \cite{weller1997scheduling,DaiPrabhakar2000}, we will characterize the set of problem parameters for which $\Gamma_{\mathrm{coh}}\geq 0.75$ so that $(2C-1)^+\geq 0.5.$ 

Our study of $\Gamma_{\mathrm{coh}}$ as a function of the problem parameters will be carried out in three steps. 
\begin{itemize}
    \item First, we will show that $C^B$ is an increasing function of $B$ for a fixed $\lambda,\mu.$
    \item Next, we will show that $C^B$ approaches $C^B$ exponentially fast in $B,$ This implies that one can achieve $C^\infty$ with a small buffer size. 
    \item Finally, by numerically calculating the stationary distribution of the reference chain presented in Section~\ref{sec: achievable}, we will characterize $\Gamma_{\mathrm{coh}}$ as a function of $\lambda, \mu, B$. The reasons will show that $\Gamma_{\mathrm{coh}}\geq 0.5$ for practical values of the system parameters. The numerical results will also show that $C^\infty$ can be achieved with relatively small buffer sizes as the theory suggests.
\end{itemize}

We start by establishing the monotonicity of $C^B,$ which follows from a standard coupleing argument.
\begin{proposition}
\label{prop:Cu_monotone_B}
For $B_1\le B_2$, we have $C^{B_1}\le C^{B_2}$. Moreover,
$
C^{B}\uparrow C^\infty$ as $B\to\infty.$$
$
\end{proposition}

\begin{proof}
Fix i.i.d.\ randomness $\{S(t),A(t)\}$ and the binomial thinning coins across time slots.
Run $\tilde L^{B_1}$ and $\tilde L^{B_2}$ with the same randomness and the same initial state.
The one-step update is monotone in the state and the truncation level, hence
$\tilde L^{B_1}(t)\le \tilde L^{B_2}(t)$ for all $t$ almost surely.
Therefore, $\mathbf{1}\{\tilde L^{B_1}(t)=0\}\ge \mathbf{1}\{\tilde L^{B_2}(t)=0\}$ for all $t$,
and taking stationary expectations yields $\pi^{B_1}_0\ge \pi^{B_2}_0$, i.e.\ $C^{B_1}\le C^{B_2}$.
The limit $C^{B}\uparrow C^\infty$ follows by coupling $\tilde L^{B}(t)=\min\{B,\tilde L^\infty(t)\}$
and applying monotone convergence to stationary expectations.
\end{proof}

The next theorem show that $C^B$ approaches $C^\infty$ exponentially fast.

\begin{theorem}[Exponential rate in $B$]
\label{thm:exp_rate_in_B}
Fix $\lambda\in(0,1)$ and $\mu\in(0,1)$.
Let $\{L(t)\}$ be the infinite-buffer reference chain,
and let $\{L^B(t)\}$ be its truncation at level $B$.
Let $\pi$ and $\pi^B$ denote their stationary distributions, and set
\[
C^\infty:=1-\pi_0,
\qquad
C^B:=1-\pi_0^B.
\]
Then there exist constants $\theta>0$ and $M<\infty$ (depending on $(\lambda,\mu)$ but not on $B$) such that
\begin{equation}
\label{eq:CB_exp_rate_correct}
0\le C^\infty-C^B \le M e^{-\theta B},
\qquad \forall B\ge 1.
\end{equation}
In particular, $C^B\uparrow C^\infty$ and the convergence is exponentially fast in $B$.
\end{theorem}

\begin{proof}
The proof is based on regeneration at state $0$.

\medskip
\noindent
\textit{Step 1:} We exploit the relationship between return times and stationary probabilities.
Let
\[
\tau:=\inf\{t\ge 1: L(t)=0\},
\qquad
\tau^B:=\inf\{t\ge 1: L^B(t)=0\},
\]
with both chains started from $L(0)=L^B(0)=0$ and coupled using the same arrivals,
service-attempt coins, and decoherence-thinning randomness.

For a positive recurrent irreducible Markov chain,
the stationary probability of state $0$ is the reciprocal of the mean return time to $0$.
Hence
\[
\pi_0=\frac{1}{\mathbb E_0[\tau]},
\qquad
\pi_0^B=\frac{1}{\mathbb E_0[\tau^B]}.
\]
Therefore
\[
\pi_0^B-\pi_0
=
\frac{\mathbb E_0[\tau]-\mathbb E_0[\tau^B]}
{\mathbb E_0[\tau]\;\mathbb E_0[\tau^B]}.
\]
Since both return times are at least $1$, the denominator is at least $1$, so
\[
0\le \pi_0^B-\pi_0 \le \mathbb E_0[\tau-\tau^B].
\]
Thus, our goal is to bound $E_0[\tau-\tau^B].$

\medskip
\noindent
\textit{Step 2:} Now, we relate the return times to zero of the finite and infinite-buffer reference chains.
Due to the coupling mentioned in Step 1 and monotonicity of the queue update rule, we have
\[
L^B(t)\le L(t)\qquad\text{for all }t\ge 0 \quad \text{a.s.}
\]
Hence
\[
\tau^B\le \tau \qquad \text{a.s.}
\]

Now define the overflow event
\[
E_B:=\left\{\max_{0\le t<\tau} L(t)\ge B+1\right\}.
\]
If $E_B$ does not occur, then the infinite-buffer chain never exceeds level $B$
before returning to $0$, so the truncation is never activated and the two chains evolve identically throughout the cycle. Therefore
\[
\tau^B=\tau \qquad \text{on } E_B^c.
\]
Thus
\begin{equation}
\label{eq:cycle_diff_on_overflow}
0\le \tau-\tau^B \le \tau\,\mathbf 1_{E_B}.
\end{equation}

\medskip
\noindent
\textit{Step 3:} Next, we argue that the overflow cycles in the infinite-buffer system are exponentially rare.
For the infinite-buffer reference chain, a standard Foster--Lyapunov drift condition holds outside a finite set. By applying the exponential bound for hitting times in \cite{hajek1982hitting}, this implies that the return time to a finite set has an exponential moment. In particular, the stationary distribution of the infinite-buffer chain has an exponential tail: 
there exist constants $\vartheta>0$ and $M_0<\infty$ such that the stationary distribution $\pi$ of the infinite-buffer chain satisfies
\begin{equation}
\label{eq:stationary_tail_exp}
\pi\{L\ge m\}\le M_0 e^{-\vartheta m},
\qquad \forall m\ge 0.
\end{equation}

Now use the standard regenerative occupation identity for the chain started from $0$:
for any set $A$,
\[
\mathbb E_0\!\left[\sum_{t=0}^{\tau-1}\mathbf 1\{L(t)\in A\}\right]
=
\frac{\pi(A)}{\pi_0}.
\]
Applying this with $A=\{B+1,B+2,\dots\}$ gives
\[
\mathbb E_0\!\left[\sum_{t=0}^{\tau-1}\mathbf 1\{L(t)\ge B+1\}\right]
=
\frac{\pi\{L\ge B+1\}}{\pi_0}.
\]
Thus,
\begin{equation}
\label{eq:overflow_prob_bound}
\mathbb P_0(E_B)
\le
\frac{\pi\{L\ge B+1\}}{\pi_0}
\le
\frac{M_0}{\pi_0}e^{-\vartheta(B+1)}.
\end{equation}
Thus the probability of an overflow cycle is exponentially small in $B$.

\medskip
\noindent
\textit{Step 4}: We now complete the proof.
Using \eqref{eq:cycle_diff_on_overflow} and Cauchy--Schwarz,
\[
\mathbb E_0[\tau-\tau^B]
\le
\mathbb E_0[\tau\,\mathbf 1_{E_B}]
\le
\big(\mathbb E_0[\tau^2]\big)^{1/2}\,\mathbb P_0(E_B)^{1/2}.
\]
The same Lyapunov argument mentioned earlier will also show that $E_0(\tau^2)<\infty.$ Hence, from \eqref{eq:overflow_prob_bound}, there exist constants $\theta_1>0$ and $M_1<\infty$ such that
\begin{equation}
\label{eq:cycle_length_diff_exp}
0\le \mathbb E_0[\tau-\tau^B]\le M_1 e^{-\theta_1 B}.
\end{equation}
From Step (1), 
using \eqref{eq:cycle_length_diff_exp},
\[
0\le \pi_0^B-\pi_0 \le M_1 e^{-\theta_1 B}.
\]
Finally, since
\[
C^\infty-C^B=(1-\pi_0)-(1-\pi_0^B)=\pi_0^B-\pi_0,
\]
we obtain \eqref{eq:CB_exp_rate_correct}.
\end{proof}

\subsection{Numerical results}
\label{sec:numerical_results}

We numerically compute the stationary distribution $\pi^{(B)}$ of the reference
Markov chain $\tilde L_u(t)\in\{0,1,\dots,B\}$
with service attempt probability $p_u=\lambda_u$. The steady-state availability is
$C_u^B=1-\pi^{(B)}_0$, and a lower bound on the fraction of the capacity region achievable by our algorithm is $ min_{e=(u,v)\in E}(C_u^B+C_v^B-1)^+$. Considering a node $u$ with the smallest value of $C_u^B$, and dropping the subscript, $\Gamma_{\mathrm{coh}}$ is further lower bounded by $(2C^B-1)^+.$ 

\begin{figure}[htbp]
         \centering
    \includegraphics[width=\linewidth]{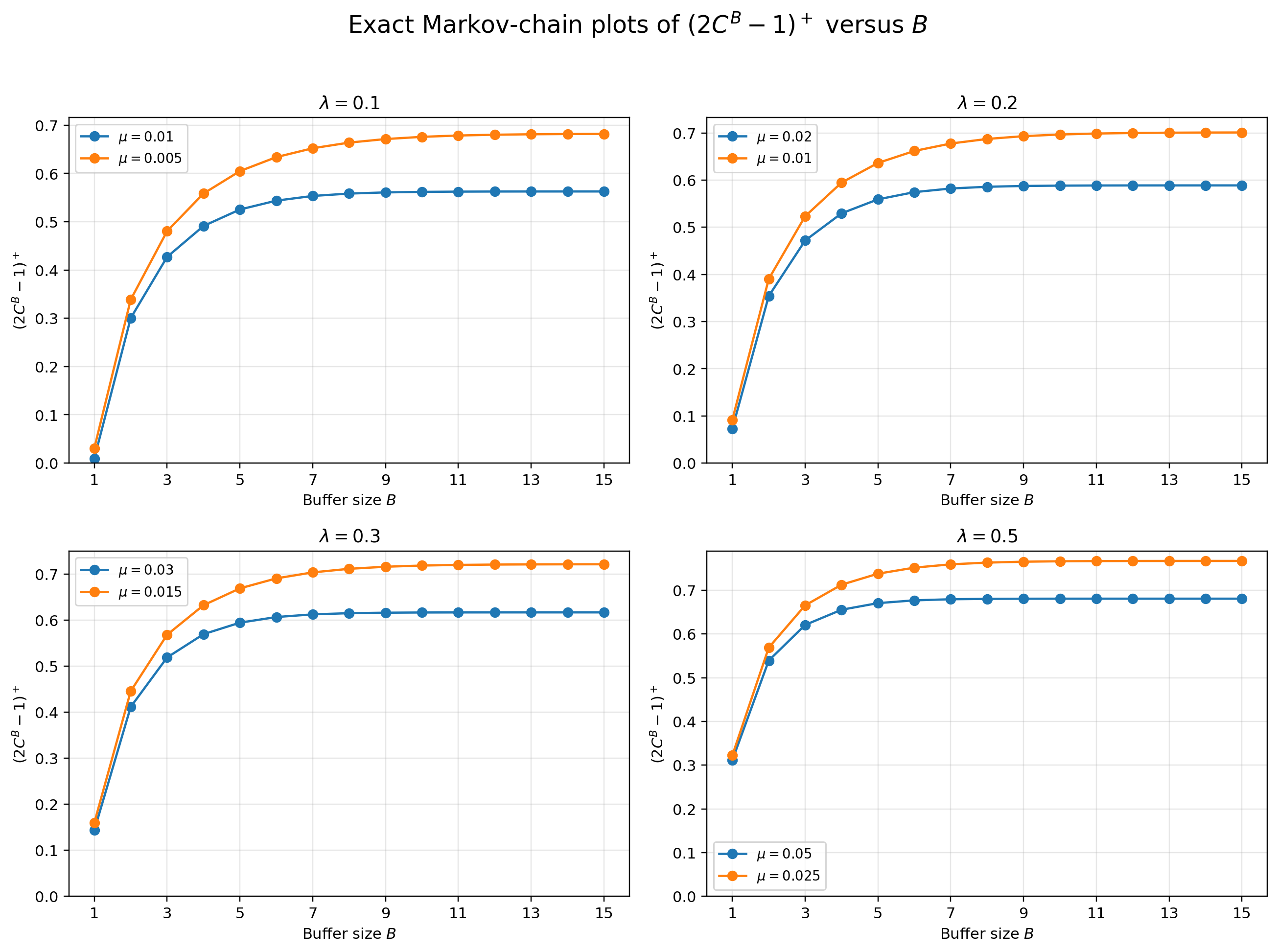}
    \caption{Results for Our Algorithm.}
    \label{fig:Algorithm I}
\end{figure}

Our numerical experiments are presented for these set of parameter values:
$\lambda \in [0.1,0.5]$, $\mu \in \{0.05\lambda,\,0.1\lambda\}$, and
buffer sizes $B \in [5,25].$ Since viable quantum switches are still under development, it is difficult to know what suitable switch parameters are consistent with quantum networking
hardware in the next five years or so. However, based on the available literature, these ranges of parameters seem feasible for the following reasons:

\begin{itemize}
\item \emph{Entanglement generation rates ($\lambda$).}
State-of-the-art quantum network testbeds achieve entanglement generation rates ranging from a few Hz to
tens of Hz over metropolitan distances, with projections toward
$\mathcal{O}(10^2)$ Hz in the near term \cite{pompili2021realizing,wehner2018quantum}.
When normalized to a time slot corresponding to a swap or scheduling epoch,
this corresponds to per-slot arrival probabilities in the range
$\lambda \in [0.1,0.5]$.

\item \emph{Decoherence rates ($\mu$).}
Quantum memories in current systems exhibit coherence times ranging from
milliseconds to seconds depending on the platform \cite{lvovsky2009optical,humphreys2018deterministic}.
Given entanglement generation times on the order of $10$--$100$ ms,
this yields a per-slot decoherence probability that is typically an order of
magnitude smaller than the arrival probability, i.e.,
$\mu \approx (0.05$--$0.1)\lambda$, consistent with the regimes considered here.

\item \emph{Buffer sizes ($B$).}
Near-term quantum repeaters are expected to support a small number of
simultaneously stored entangled pairs per link due to hardware constraints such
as limited memory qubits and control overhead \cite{wehner2018quantum,van2013repeaters}.
Experimental platforms currently demonstrate storage of a handful of qubits,
with projections toward tens of qubits per node in early deployments.
Thus, buffer sizes in the range $B \in [5,25]$ appear to be feasible for
near-term systems.
\end{itemize}

Figure~\ref{fig:Algorithm I} shows the numerical results for our algorithm.
The following observation can be made from the figures.
for any fixed $(\lambda,\mu)$,
$\Gamma_{\mathrm{coh}}$ increases rapidly with $B$ and saturates after a modest buffer size. This is consistent with the theory in Section~\ref{sec: achievable}.

\section{Conclusion}

In this paper, we considered the problem of scheduling in a quantum switch with stochastic entanglement generation, finite memory, and decoherence. Our objective was to design a scheduling algorithm with polynomial computational complexity that stabilizes a nontrivial fraction of the capacity region. Our results provide a tractable alternative to throughput-optimal dynamic-programming-based scheduling for quantum switches and help quantify the impact of entanglement generation rates, decoherence, and memory size on achievable throughput. Future work includes  extending the framework to richer entanglement-request models, and developing delay and queue-length performance guarantees.

\vskip 1em

\noindent\textbf{Acknowledgments:} 
\begin{itemize}
    \item Research supported by NSF Grant CCF 22-07547.
    \item The author gratefully acknowledges the use of ChatGPT and Gemini for brainstorming, literature search, editorial assistance, and obtaining numerical results. 
\end{itemize}

\bibliographystyle{abbrv}
\bibliography{switch}

\end{document}